\documentclass[10pt]{amsart}
\usepackage{amssymb}
\newtheorem{lemma}{Lemma}
\newtheorem{corollary}{Corollary}
\newtheorem{theorem}{Theorem}
\newtheorem{remark}{Remark}

\begin{document}

\title [Tur\'an type inequalities for  Mittag-Leffler functions]{Tur\'an type inequalities for classical and generalized Mittag-Leffler functions\\}%

\author[ K. Mehrez, S. M. Sitnik]{KHALED MEHREZ\;and\; Sergei M. Sitnik }

\address{Khaled Mehrez.\newline D\'epartement de Math\'ematiques ISSAT Kasserine, Universit\'e de Kairouan, Tunisia.}
 \email{k.mehrez@yahoo.fr}

\address{Sergei M. Sitnik.\newline Voronezh Institute of the Ministry of Internal Affairs of Russia, Patriotov  pr.,  53, Voronezh, 394065, Russia;\newline
and\newline
 Peoples' Friendship University of Russia, M. - Maklaya str., 6, Moscow, 117198, Russia.}
\email{pochtaname@gmail.com}

\begin{abstract}
 In this paper  some Tur\'an type inequalities for classical and generalized Mittag--Leffler functions are considered. The method is based on proving  monotonicity for special ratio of sections for series of such  functions.
Some applications are considered to   Lazarevi\'c--type and Wilker--type inequalities for classical and generalized Mittag--Leffler functions.
\end{abstract}
\maketitle
\noindent\textbf{keywords:} Mittag--Leffler functions; Tur\'an type inequalities; Lazarevi\'c--type inequalities; Wilker--type inequalities.  \\
\\
\noindent\textbf{MSC (2010):} 33E12, 26D07.
\section{\textbf{Introduction}}
We use a definition of Mittag--Leffler function by its series
\begin{equation}\label{ML}
E_{\alpha, \beta}(z)=\sum_{n=0}^{\infty}\frac{z^n}{\Gamma(\alpha n+\beta)},\;\; z\in\mathbb{C},\;\alpha, \beta\in\mathbb{C},\;	 Re(\alpha)>0,\; Re(\beta)>0.
\end{equation}
This function was first introduced by G.~Mittag--Leffler in 1903 for $\alpha=1$ and by A.~Wiman in 1905 for the general case (\ref{ML}).
For the mathematical theory and properties of Mittag--Leffler functions cf. \cite{Dzh}, \cite{KRM}, \cite{NIST}.

First applications of the function  (\ref{ML}) by  Mittag--Leffler and Viman were in complex function theory (non--trivial examples of entire functions with fractional growth characteristics such as order and generalized summation methods). But its really important applications were found in 20--th century for  fractional integral and differential equations. The most known result in this field is an explicit formula for the resolvent of Riemann--Liouville fractional integral proved by E.~Hille and J.~Tamarkin in 1930. On this and similar formulas many results are based still for solving fractional integral and differential equations. For numerous applications of the Mittag--Leffler function  to fractional calculus cf. \cite{SKM}, \cite{KRM}, \cite{Die}, \cite{KST}, \cite{Plb}.
Due to many useful applications it  was crowned by R.~Gorenflo and F.~Mainardi in \cite{GM} as a \textsl{"Queen function of Fractional Calculus"}!
Besides fractional calculus  the Mittag--Leffler function also plays an important role in various branches of applied mathematics and engineering sciences, such as chemistry, biology,  statistics, thermodynamics, mechanics, quantum physics, informatics, signal processing  and others.

There are further related generalizations of the  Mittag--Leffler function, namely Wright and Fox functions. Wright functions are defined in the same way as (\ref{ML}) but with more gamma--functions both in numerator and denominator, sometimes these functions are called "multi--indexed Mittag--Leffler functions", cf. \cite{Kir1},\cite{Kir2}, \cite{SM}. Fox function is defined by the Mellin transform, cf. \cite{KS}, \cite{MS}, \cite{MSH}, \cite{SM}, \cite{SGG}, it has also important applications e.g. in fractional diffusion theory, cf.  \cite{Koch1}, \cite{Koch2}, \cite{Koch3}.

By means of the series representation Prabhakar in  \cite{PR} studied the function
\begin{equation}\label{ML2}
E_{\alpha, \beta}^\gamma(z)=\sum_{n=0}^{\infty}\frac{(\gamma)_n z^n}{n!\Gamma(\alpha n+\beta)},\;\; z\in\mathbb{C},\;\alpha, \beta, \gamma\in\mathbb{C},\;	 Re(\alpha)>0,\; Re(\beta)>0,\; Re(\gamma)>0.
\end{equation}
where $(\gamma)_n$ is the Pochhammer symbol given by
$$(\gamma)_n=\frac{\Gamma(\gamma+n)}{\Gamma(\gamma)},\;(\gamma)_0=1.$$
In 2007 Shukla and Prajapati \cite{SP} studied another generalization with four parameters $E_{\alpha,\beta}^{\gamma,q}(z)$ which is defined for  $z, \alpha, \beta, \gamma\in\mathbb{C},\;	 Re(\alpha)>0,\; Re(\beta)>0,\; Re(\gamma)>0$ and $q\in(0,1)\cup\mathbb{N}$ as
\begin{equation}\label{ML3}
E_{\alpha, \beta}^{\gamma,q}(z)=\sum_{n=0}^{\infty}\frac{(\gamma)_{qn} z^n}{n!\Gamma(\alpha n+\beta)}.
\end{equation}
Note that
$$E_{1, 1}^{1,1}(z)=e^z,\;E_{\alpha,1 }^{1,1}(z)=E_\alpha(z),\;E_{\alpha, \beta}^{1,1}(z)=E_{\alpha, \beta}(z),\;\textrm{and}\;\;E_{\alpha, \beta}^{\gamma,1}(z)=E_{\alpha, \beta}^{\gamma}(z).$$
These interesting generalizations are special cases of generalized Mittag--Leffler function, cf. \cite{Kir1}--\cite{Kir2}.

An important result which initiated  a new field of researches on inequalities for special functions was proved   by  Paul Tur\'an, it is:

\begin{equation}
 [P_n(x)]^2-P_{n+1}(x)P_{n-1}(x) \ge 0,
\end{equation}
where $-1<x<1$, $n\in\mathbb{N}$ and $P_n(\cdot)$ stands  for the classical Legendre polynomial. This inequality was
published by Tur\'an in 1950 in \cite{Tur} but proved earlier in 1946  in a letter to Szeg\"o. Since the publication of the above  Tur\'an inequality  in 1948 by Szeg\"o \cite{Sze} many authors derived  results of such type for classical orthogonal polynomials and different special functions.
The Tur\'an type inequalities
now have an extensive literature and some of the results have been applied successfully to different problems
in information theory, economic theory, biophysics, probability and statistics. For more details cf.  \cite{Bar1}, \cite{Bar2}, \cite{BG},  \cite{MRS}, \cite{MSK}, \cite{SB}. The results on Tur\'an type inequalities are closely connected with log--convex and log--concave functions, cf.  \cite{Kar1}, \cite{Kar2}, \cite{KaSi1}, \cite{KaSi2}.

Now Tur\'an--type inequalities are proved for different classes of special functions: Kummer hypergeometric functions (cf. \cite{BG}, \cite{MS1}, \cite{MS2}, \cite{MS3}, \cite{Sita1}, \cite{Sita2}, \cite{Sita3}), Gauss hypergeometric functions (cf. \cite{KaSi1}, \cite{KaSi2}, \cite{MS1}, \cite{MS2}, \cite{MS3}), different types of Bessel functions (cf. \cite{Bar1}, \cite{Bar2}),
Dunkl kernel and $q$--Dunkl kernel (cf. \cite{MSK}), $q-$Kummer hypergeometric functions (cf. \cite{MS4}, \cite{MS5}, \cite{MS6}) and some others.

This paper is a continuation of some line of authors results.
In 1990 one of the authors studied inequalities for sections of series for exponential function in \cite{Sita1}. Among other results in \cite{Sita1} a conjecture was proposed on monotonicity of ratios for Kummer   hypergeometric function, cf. also \cite{Sita2}--\cite{Sita3}. This conjecture was proved recently by the authors in \cite{MS1}--\cite{MS2}, cf. also \cite{MS3}--\cite{MS4}. After that $q$--versions of these results were proved in \cite{MS5}--\cite{MS7}.

The paper is organized as follows. In section 2 we collect some lemmas. In section 3 we give some Tur\'an type inequalities for Mittag--Leffler functions. Moreover, we prove  monotonicity of ratios for sections of series of Mittag--Leffler functions, the result is also closely connected with Tur\'an-type inequalities. In section 4 we deduce new inequalities
of Lazarevi\'c--type and Wilker--type  for  Mittag--Leffler functions. In sections 5 and 6 similiar results are proved for generalized Mittag--Leffler functions \eqref{ML2}--\eqref{ML3}, which demonstrate that the technics of this paper is applicable for these functions too. At the end of the paper two unsolved problems are included.

\section{\textbf{Useful lemmas}}

We need the following two useful lemmas proved in \cite{BK}, \cite{PV}.

\begin{lemma}\label{l1}
Let $(a_{n})$ and $(b_{n})$ $(n=0,1,2...)$ be real numbers, such that $b_{n}>0,\;n=0,1,2,...$ and $\left(\frac{a_{n}}{b_{n}}\right)_{n\geq 0}$ is increasing (decreasing), then $\left(\frac{a_{0+}...+a_{n}}{b_{0}+...+b_{n}}\right)_{n}$ is also increasing (decreasing).
\end{lemma}

\begin{lemma}\label{l2}
Let $(a_{n})$ and $(b_{n})$ $(n=0,1,2...)$ be real numbers and
let the power series $A(x)=\sum_{n=0}^{\infty}a_{n}x^{n}$ and $B(x)=\sum_{n=0}^{\infty}b_{n}x^{n}$
be convergent for $|x|<r$. If $b_{n}>0,\, n=0,1,2,...$ and  the
sequence $\left(\frac{a_{n}}{b_{n}}\right)_{n\geq0}$is (strictly)
increasing (decreasing) , then the function $\frac{A(x)}{B(x)}$ is also
(strictly) increasing on $[0,r)$.
\end{lemma}

\section{\textbf{Tur\'an type inequalities for   Mittag--Leffler functions}}

Our first  result is the next theorem.

\begin{theorem} Let $\alpha, \beta>0.$  Then the following assertions are true.\\

\noindent	\textbf{a.} The function $\beta\mapsto \mathbb{E}_{\alpha,\beta}(z)= \Gamma(\beta) E_{\alpha,\beta}(z)$ is log--convex on $(0,\infty).$\\

\noindent	\textbf{b.}  The following Tur\'an type inequality
	\begin{equation}\label{6}
	\mathbb{E}_{\alpha,\beta}(z)\mathbb{E}_{\alpha,\beta+2}(z)-\mathbb{E}^2_{\alpha,\beta+1}(z)\geq0
	\end{equation}
	holds for all $z\in(0,\infty).$

In particular, the following inequality
	\begin{equation}\label{66}
	(e^z-1)^2\leq2e^z(e^z-1-z)
	\end{equation}
	is valid for all $z>0.$\\

\noindent	\textbf{c.} For $n\in\mathbb{N}$  define the function $E^{n}_{\alpha, \beta}(z)$ by
$$E^{n}_{\alpha, \beta}(z)=E_{\alpha, \beta}(z)-\sum_{k=0}^{n}\frac{z^k}{\Gamma(\alpha k+\beta)}=\sum_{k=n+1}^{\infty}\frac{z^k}{\Gamma(\alpha k+\beta)}.$$
Then the following Tur\'an type inequality
\begin{equation}\label{8}
E^n_{\alpha,\beta}(z)E^{n+2}_{\alpha,\beta}(z)\leq [E^{n+1}_{\alpha,\beta}(z)]^2.
\end{equation}
is valid for all $n\in\mathbb{N}$, and $\alpha, \beta>0$ and $z>0.$
\end{theorem}

\begin{proof}
\textbf{a.} For log--convexity of $\beta\mapsto\mathbb{E}_{\alpha,\beta}(z)$ we observe that it is enough to show the log--convexity of each individual term and to use the fact that the sum of log--convex functions
is log--convex too. Thus, we just need to show that for each $k\geq0$ we have
\[\frac{\partial^{2}}{\partial\beta^{2}}\log\left[\frac{\Gamma(\beta)}{\Gamma(\beta+\alpha k)}\right]=\psi^{'}(\beta)-\psi^{'}(\beta+\alpha k) \ge 0,\]
where $\psi(x)=\Gamma^{\prime}(x)/\Gamma(x)$ is the so--called digamma function. But $\psi$ is known to be concave, and consequently the function  $\beta\mapsto \frac{\Gamma(\beta)}{\Gamma(\beta+k\alpha)}$ is log--convex on $(0,\infty).$\\

\textbf{b.} Since the function $\beta\mapsto \mathbb{E}_{\alpha,\beta}(z)$ is log--convex, then for all $\beta_1, \beta_2>0,\;z>0$ and $t\in[0,1]$ we have
$$ \mathbb{E}_{\alpha, t\beta_1+(1-t)\beta_2}(z)\leq  [\mathbb{E}_{\alpha,\beta_1}(z)]^t[\mathbb{E}_{\alpha,\beta_2}(z)]^{1-t}.$$
Now choosing $t=1/2,\;\beta_1=\beta,\;\beta_2=\beta+2$ we conclude that (\ref{6}) holds. To prove the inequality (\ref{66}) choose $\alpha=\beta=1$ in(\ref{6}) and use a recurrence relation from (\cite{MSH}, Theorem 5.1)
 $$E_{\alpha,\beta}(z)=z E_{\alpha,\alpha+\beta}(z)+\frac{1}{\Gamma(\beta)}.$$

\textbf{c.} Let $n\in\mathbb{N},$ from the definition of the function $E_n(\alpha,\beta,z),$ we have
$$E^{n}_{\alpha,\beta}(z)=E^{n+1}_{\alpha,\beta}(z)+\frac{z^{n+1}}{\Gamma(\beta+(n+1)\alpha)}$$
and
$$E^{n+2}_{\alpha,\beta}(z)=E^{n+1}_{\alpha,\beta}(z)-\frac{z^{n+2}}{\Gamma(\beta+(n+2)\alpha)}.$$
It implies that
$$E^{n}_{\alpha,\beta}(z)E^{n+2}_{\alpha,\beta}(z)-[E^{n+1}_{\alpha,\beta}(z)]^2=$$
$$=E^{n+1}_{\alpha,\beta}(z)\left(\frac{z^{n+1}}{\Gamma(\beta+(n+1)\alpha)}-\frac{z^{n+2}}{\Gamma(\beta+(n+2)\alpha)}\right)-\frac{z^{2n+3}}{\Gamma(\beta+(n+1)\alpha)\Gamma(\beta+(n+2)\alpha)}$$
$$=\sum_{k=n+3}^{\infty}\left[\frac{1}{\Gamma(\beta+(n+1)\alpha)\Gamma(\beta+k\alpha)}-\frac{1}{\Gamma(\beta+(n+2)\alpha)\Gamma(\beta+(k-1)\alpha)}\right]z^{k+n+1}$$
$$=\sum_{k=n+3}^{\infty}\frac{A_k(\alpha,\beta)}{\Gamma(\beta+(n+1)\alpha)\Gamma(\beta+k\alpha)\Gamma(\beta+(n+2)\alpha)\Gamma(\beta+(k-1)\alpha)}z^{k+n+1},$$
where
$$A_k(\alpha,\beta)=\Gamma(\beta+(n+2)\alpha)\Gamma(\beta+(k-1)\alpha)-\Gamma(\beta+(n+1)\alpha)\Gamma(\beta+k\alpha).$$
Due to log--convexity of $\Gamma(x)$, the ratio $x\mapsto\frac{\Gamma(x+a)}{\Gamma(x)}$ is increasing on $(0,\infty)$ when $a>0.$ It implies the following inequality
\begin{equation}\label{88}
\frac{\Gamma(x+a)}{\Gamma(x)}\leq\frac{\Gamma(x+a+b)}{\Gamma(x+b)}.
\end{equation}
holds for all $a,b>0.$ For $n\geq0$ and $k\geq n+3$, let $x=\beta+(n+1)\alpha,\;a=\alpha\;,\beta=\alpha(k-(n+2))$ in (\ref{88}) we conclude that $A_k(\alpha,\beta)\leq0$, and so the Tur\'an type inequality (\ref{8}) is proved.
\end{proof}

\begin{corollary}The following Tur\'an type inequality
\begin{equation}\label{09}
E_{\alpha,\beta+(n+1)\alpha}(z)E_{\alpha,\beta+(n+3)\alpha}(z)\leq [E_{\alpha,\beta+(n+2)\alpha}(z)]^2,
\end{equation}
holds for all  $n\in\mathbb{N}$, and $\alpha, \beta\geq0$ and $z>0.$
\end{corollary}
\begin{proof} In \cite{MSH}  the following formula for Mittag--Leffler functions was  proved
\begin{equation}\label{9}
z^n E_{\alpha,\beta+n\alpha}(z)=E_{\alpha,\beta}(z)-\sum_{k=0}^{n-1}\frac{z^k}{\Gamma(\beta+k\alpha)},
\end{equation}
\end{proof}
it holds for all $\alpha>0,\;\beta>0$ and $n\in\mathbb{N}.$ Using (\ref{8}) and (\ref{9}) we conclude that (\ref{09}) holds.

\begin{theorem} Let $\alpha,\beta>0\;$ and $n\in\mathbb{N}.$ Then the function $h_{n}(\alpha,\beta,z)$ defined by
$$x\mapsto h_{n}(\alpha,\beta,z)=\frac{E^n_{\alpha,\beta}(z)E^{n+2}_{\alpha,\beta}(z)}{[E^{n+1}_{\alpha,\beta}(z)]^2}$$
is increasing on $(0,\infty).$ So the following Tur\'an type inequality
\begin{equation}\label{999}
\frac{\Gamma^2(\beta+(n+2)\alpha)}{\Gamma(\beta+(n+1)\alpha)\Gamma(\beta+(n+3)\alpha)}[E^{n+1}_{\alpha,\beta}(z)]^2\leq E^n_{\alpha,\beta}(z)E^{n+2}_{\alpha,\beta}(z),
\end{equation}
holds for all $\alpha,\beta>0,\;n\in\mathbb{N}$ and $z>0.$ The constant in LHS of inequality (\ref{999}) is sharp.
\end{theorem}

\begin{proof} From the Cauchy series product we get
$$h_{n}(\alpha,\beta,z)=\frac{\sum_{k=0}^{\infty}
\left(\sum_{j=0}^k\frac{1}{\Gamma(\beta+(n+1+j)\alpha)
\Gamma(\beta+(n+3+k-j)\alpha)}\right)z^{2n+2+k}}
{\sum_{k=0}^{\infty}\left(\sum_{j=0}^k\frac{1}
{\Gamma(\beta+(n+2+j)\alpha)\Gamma(\beta+(n+2+k-j)\alpha)}\right)z^{2n+2+k}}
=
$$
$$
=\frac{\sum_{k=0}^{\infty}\sum_{j=0}^k a_j(\alpha,\beta)z^{2n+2+k}}{\sum_{k=0}^{\infty}\sum_{j=0}^k b_j(\alpha,\beta)z^{2n+2+k}}.$$
Now we consider the sequence $(U_j)_{j\geq0}$ defined by
$$U_j=\frac{\Gamma(\beta+(n+2+j)\alpha)\Gamma(\beta+(n+2+k-j)\alpha)}{\Gamma(\beta+(n+1+j)\alpha)\Gamma(\beta+(n+3+k-j)\alpha)}.$$
Thus
\begin{eqnarray*}
\frac{U_{j+1}}{U_j}=\frac{\Gamma(\beta+(n+3+j)\alpha)
\Gamma(\beta+(n+1+j)\alpha)}{\Gamma^2(\beta+(n+2+j)\alpha)}\cdot\\
\cdot \frac{\Gamma(\beta+(n+1+k-j)\alpha)\Gamma(\beta+(n+3+k-j)\alpha)}
{\Gamma^2(\beta+(n+2+k-j)\alpha)}=\\
=\frac{\Gamma(\beta_1+(n+3)\alpha)\Gamma(\beta_1+(n+1)\alpha)}
{\Gamma^2(\beta_1+(n+2)\alpha)}.\frac{\Gamma(\beta_2+(n+1)\alpha)
\Gamma(\beta_2+(n+3)\alpha)}{\Gamma^2(\beta_2+(n+2)\alpha)},
\end{eqnarray*}
where $\beta_1=\beta+j\alpha$ and $\beta_2=\beta+(k-j)\alpha.$ And again using the Tur\'an type inequality (\ref{88}) we conclude that $\frac{U_{j+1}}{U_j}\geq1$ and consequently the sequence $(U_j)_{j\geq0}$ is increasing. So from lemma \ref{l1} we conclude that $\frac{\sum_{j=0}^{k}a_j(\alpha,\beta)}{\sum_{j=0}^{k}b_j(\alpha,\beta)}$ is increasing. Therefore, the function $x\mapsto h_{n}(\alpha,\beta,z)$ is also increasing on $(0,\infty)$ by lemma \ref{l2}. Finally,
\[
\lim_{x\rightarrow0}h_{n}(\alpha,\beta,z)=\frac{\Gamma^2(\beta+(n+2)\alpha)}{\Gamma(\beta+(n+1)\alpha)\Gamma(\beta+(n+3)\alpha)}.\]
And it follows that the next constant $\frac{\Gamma^2(\beta+(n+2)\alpha)}{\Gamma(\beta+(n+1)\alpha)\Gamma(\beta+(n+3)\alpha)}$ is the best possible for which the inequality (\ref{999}) holds for all $\alpha,\beta>0,\;n\in\mathbb{N}$ and $z>0.$
\end{proof}
\begin{theorem} Let $\alpha>0,\beta_1,\beta_2>1.$ If $\beta_1<\beta_2, \;(\beta_2<\beta_1)$, then the function $z\mapsto E_{\alpha,\beta_1}(z)\big/E_{\alpha,\beta_2}(z)$ is increasing (decreasing) on $(0,\infty).$ So, the following Tur\'an type inequalities
\begin{equation}\label{7777}
E_{\alpha,\beta_2}(z)E_{\alpha,\beta_1-1}(z)-E_{\alpha,\beta_1}(z)E_{\alpha,\beta_2-1}(z)+(\beta_2-\beta_1)E_{\alpha,\beta_1}(z)E_{\alpha,\beta_2}(z)\geq0,
\end{equation}
holds for all $\alpha,\beta_1,\beta_2>0$ such that $\beta_2>\beta_1.$ In particular, the Tur\'an type inequality
\begin{equation}\label{8888}
E_{\alpha,\beta}(z)E_{\alpha,\beta+2}(z)-E_{\alpha,\beta+1}^2(z)+(\beta+1)E_{\alpha,\beta+1}(z)E_{\alpha,\beta+2}(z)\geq0,
\end{equation}
is valid for all $\alpha,\beta,z>0.$
\end{theorem}
\begin{proof}
By using the power-series representation of the Mittag-Leffler function $E_{\alpha,\beta}(z)$, we have
$$E_{\alpha,\beta_1}(z)\big/ E_{\alpha,\beta_2}(z)=\sum_{k=0}^\infty\frac{z^k}{\Gamma(\beta_1+k\alpha)}\Bigg/\sum_{k=0}^\infty\frac{z^k}{\Gamma(\beta_2+k\alpha)}.$$
In view of Lemma \ref{l2}  we need to study the monotonicity of sequence $(u_k)_{k\geq0}$ defined by
$$u_k=\frac{\Gamma(\beta_2+k\alpha)}{\Gamma(\beta_1+k\alpha)},\;k\geq0.$$
Thus
$$\frac{u_{k+1}}{u_k}=\frac{\Gamma(\beta_2+\alpha+k\alpha)\Gamma(\beta_1+k\alpha)}{\Gamma(\beta_2+k\alpha)\Gamma(\beta_1+\alpha+k\alpha)}.$$
In the case  $\beta_1>\beta_2>1$, we let $x=\beta_2+k\alpha,\;a=\alpha$ and $b=\beta_1-\beta_2>0$ in (\ref{88}) we obtain that
$$\frac{u_{k+1}}{u_k}=\frac{\Gamma(\beta_2+\alpha+k\alpha)\Gamma(\beta_1+k\alpha)}{\Gamma(\beta_2+k\alpha)\Gamma(\beta_1+\alpha+k\alpha)}\leq1.$$
Consequently, $u_{k+1}\leq u_k$ for all $k\geq0$ if and only if $\beta_1>\beta_2,$ and the function $z\mapsto E_{\alpha,\beta_1}(z)\big/E_{\alpha,\beta_2}(z)$ is decreasing on $(0,\infty)$ if $\beta_1>\beta_2,$ by Lemma \ref{l2}. In the case $\beta_2>\beta_1$ we set $x=\beta_1+k\alpha,\;a=\alpha$ and $b=\beta_2-\beta_1>0$ in (\ref{88}) we conclude that $u_{k+1}\geq u_k$ for all $k\geq0.$ So the function $z\mapsto E_{\alpha,\beta_1}(z)\big/E_{\alpha,\beta_2}(z)$ is increasing on $(0,\infty)$ if $\beta_2>\beta_1,$ by means of Lemma \ref{l2}.  From the differentiation formula [\cite{MSH}, Theorem 5.1]
\begin{equation}\label{diff}
\frac{d}{dz}E_{\alpha,\beta}(z)=\frac{E_{\alpha,\beta-1}(z)-(\beta-1)E_{\alpha,\beta}(z)}{\alpha z}
\end{equation}
we get for $\beta_2>\beta_1$
$$\left[\frac{E_{\alpha,\beta_1}(z)}{ E_{\alpha,\beta_2}(z)}\right]^\prime=\frac{E_{\alpha,\beta_1-1}(z)E_{\alpha,\beta_2}(z)-E_{\alpha,\beta_2-1}(z)E_{\alpha,\beta_1}(z)+(\beta_2-\beta_1)E_{\alpha,\beta_1}(z)E_{\alpha,\beta_2}(z)}{\alpha z E_{\alpha,\beta_2}^2(z)}\geq0,$$
and with this the proof of the inequality (\ref{7777}) is done.  Finally, choosing $\beta_1=\beta+1$ and $\beta_2=\beta+2$ in the inequality (\ref{7777}) we obtain (\ref{8888}).
\end{proof}

\section{\textbf{Applications: Lazarevi\'c and Wilker--type inequalities for Mittag--Leffler functions}}

\begin{theorem}\label{t3} Let $\alpha, \beta_1,\beta_2>0$ be such that $\beta_1\geq\beta_2>1.$ Then the following inequality
\begin{equation}\label{a1}
\left[\mathbb{ E}_{\alpha,\beta_{1}}\left(z\right)\right]^{\frac{\Gamma(\beta_{1}-1)}{\Gamma(\beta_{1})}}\leq\left[\mathbb{ E}_{\alpha,\beta_{2}}\left(z\right)\right]^{\frac{\Gamma(\beta_{2}-1)}{\Gamma(\beta_{2})}}
\end{equation}
holds for all $z\in\mathbb{R}.$
\end{theorem}

\begin{proof}From part (\textbf{a.}) of the theorem 1 the function $\beta\mapsto\log \mathbb{E}_{\alpha,\beta}(z)$ is convex and hence it follows that $\beta\mapsto\log[\mathbb{ E}_{\alpha,\beta+a}(z)]-\log[\mathbb{ E}_{\alpha,\beta}(z)]$ is increasing for each $a>0.$ Thus, choosing $a=1$ we obtain that indeed the function $\beta\mapsto\frac{\mathbb{ E}_{\alpha,\beta+1}(z)}{\mathbb{E}_{\alpha,\beta}(z)}$ is increasing on $(0,\infty).$ Now, providing that $\beta_1\geq\beta_2>1$ let define the function $\Phi:\mathbb{R}\longrightarrow\mathbb{R}$ by
$$\Phi(x)=\frac{\Gamma(\beta_2)\Gamma(\beta_1-1)}{\Gamma(\beta_2-1)\Gamma(\beta_1)}\log[\mathbb{E}_{\alpha,\beta_1}(z)]-\log[\mathbb{E}_{\alpha,\beta_2}(z)].$$
From the differentiation formula (\ref{diff}) we get
\begin{equation*}
\begin{split}
\Phi^{\prime}(x)&=\frac{1}{\alpha z}\left[\frac{\Gamma(\beta_2)\Gamma(\beta_1-1)}
{\Gamma(\beta_2-1)\Gamma(\beta_1)}\frac{E_{\alpha,\beta_1-1}(z)}
{E_{\alpha,\beta_1}(z)}-\frac{E_{\alpha,\beta_2-1}(z)}
{E_{\alpha,\beta_2}(z)}+(\beta_2-1)-\frac{\Gamma(\beta_2)\Gamma(\beta_1-1)}
{\Gamma(\beta_2-1)\Gamma(\beta_1)}(\beta_1-1)\right]\\
&=\frac{\Gamma(\beta_2)}{\alpha z\Gamma(\beta_2-1)}\left[\frac{\mathbb{E}_{\alpha,\beta_1-1}(z)}
{\mathbb{E}_{\alpha,\beta_1}(z)}-\frac{\mathbb{E}_{\alpha,\beta_2-1}(z)}{\mathbb{E}_{\alpha,\beta_2}(z)}\right].
\end{split}
\end{equation*}
Since the function $\beta\mapsto\frac{\mathbb{E}_{\alpha,\beta+1}(z)}{\mathbb{ E}_{\alpha,\beta}(z)}$ is increasing on $(0,\infty)$ we derive for all $\beta_1\geq\beta_2>1$ that
$$\frac{\mathbb{E}_{\alpha,\beta_1-1}(z)}{\mathbb{ E}_{\alpha,\beta_1}(z)}\leq\frac{\mathbb{ E}_{\alpha,\beta_2-1}(z)}{\mathbb{E}_{\alpha,\beta_2}(z)}.$$
From this we conclude that the function $z\mapsto \Phi(z)$ is decreasing on $[0,\infty)$ and increasing on $(-\infty,0]$. Consequently $\Phi(z)\leq\Phi(0)=0$ for all $z\in\mathbb{R}.$ So the proof of the theorem \ref{t3} is complete.
\end{proof}
\begin{remark} Choosing $\beta_1=\beta+1, \beta_2=\beta$ in (\ref{a1}) we obtain
\begin{equation}\label{a2}
\mathbb{ E}_{\alpha,\beta+1}\left(z\right)\leq\left[\mathbb{ E}_{\alpha,\beta}\left(z\right)\right]^{\frac{\beta}{\beta-1}},\;z\in\mathbb{R}.
\end{equation}
If $\beta=3/2$ we derive the Lazarevi\'c--type inequality \cite{L} for the Mittag--Leffler function
\begin{equation}\label{a3}
\mathbb{ E}_{\alpha,\frac{5}{2}}\left(z\right)\leq [\mathbb{ E}_{\alpha,\frac{3}{2}}\left(z\right)]^3,\;z\in\mathbb{R}.
\end{equation}
\end{remark}

\begin{corollary}Let $\alpha, \beta_1,\beta_2>0$ be such that $\beta_1\geq\beta_2>1.$ Then the following inequality
\begin{equation}\label{a4}
\left[\mathbb{ E}_{\alpha,\beta_2}(z)\right]^{\frac{\beta_1-\beta_2}{\beta_2-1}}+\frac{\mathbb{ E}_{\alpha,\beta_2}(z)}{\mathbb{ E}_{\alpha,\beta_1}(z)}\geq2,
\end{equation}
holds for all $z\in\mathbb{R}.$
\end{corollary}

\begin{proof} By using the inequality (\ref{a1}) and the arithmetic--geometric mean inequality we conclude that
$$\frac{1}{2}\left(\left[\mathbb{ E}_{\alpha,\beta_2}(z)\right]^{\frac{\beta_1-\beta_2}{\beta_2-1}}+\frac{\mathbb{ E}_{\alpha,\beta_2}(z)}{E_{\alpha,\beta_1}(z)}\right)\geq \sqrt{\frac{[\mathbb{ E}_{\alpha,\beta_2}(z)]^{\frac{\beta_1-1}{\beta_2-1}}}{\mathbb{ E}_{\alpha,\beta_1}(z)}}\geq1.$$
\end{proof}

Note that with use of generalizations of AGM--inequality we may refine (\ref{a4}), on generalizations of the AGM--inequality cf. \cite{BMV}, \cite{MPF} and related Cauchy--Bunyakovsky inequality cf. \cite{SitaCB}.

\begin{remark} For $\beta_1=\beta+1, \beta_2=\beta$ in (\ref{a4}) we obtain
\begin{equation}\label{a5}
\left[\mathbb{ E}_{\alpha,\beta}(z)\right]^{\frac{1}{\beta-1}}+
\frac{\mathbb{ E}_{\alpha,\beta}(z)}{\mathbb{ E}_{\alpha,\beta+1}(z)}\geq2,\;z\in\mathbb{R}.
\end{equation}

In case $\beta=3/2$ the next Wilker--type inequality \cite{W} for the Mittag--Leffler function follows
\begin{equation}\label{a6}
[\mathbb{ E}_{\alpha,\frac{3}{2}}(z)]^2+\frac{\mathbb{ E}_{\alpha,\frac{3}{2}}(z)}{\mathbb{ E}_{\alpha,\frac{5}{2}}(z)}\geq2,\;z\in\mathbb{R}.
\end{equation}
\end{remark}

\section{Tur\'an type inequalities for the Generalized Mittag--Leffler functions}

\begin{theorem}Let Let $\alpha>0,\;\beta>0, \gamma>0$, and $q\in(0,1)\cup\mathbb{N}.$ Then the following assertions are true:\\
\textbf{a.} The function $\gamma\mapsto E_{\alpha,\beta}^{\gamma,q}(z)$ is increasing on $(0,\infty).$\\
\textbf{b.} The function $\beta\mapsto \mathbb{E}_{\alpha,\beta}^{\gamma,q}(z)=\Gamma(\beta)E_{\alpha,\beta}^{\gamma,q}(z)$ is log-convex on $(0,\infty).$\\
\textbf{c.} Let $z>0.$ Then following Tur\'an type inequalities
\begin{equation}\label{666}
\left(\mathbb{E}_{\alpha,\beta+1}^{\gamma,q}(z)\right)^2\leq \mathbb{E}_{\alpha,\beta}^{\gamma,q}(z)\mathbb{E}_{\alpha,\beta+2}^{\gamma,q}(z),
\end{equation}
holds for all $\alpha>0,\;\beta>0, \gamma>0$, and $q\in(0,1)\cup\mathbb{N}.$
\end{theorem}
\begin{proof}
\textbf{a}. Let us write
$$E_{\alpha,\beta}^{\gamma,q}(z)=\sum_{n=0}^{\infty}a_n^q(\alpha,\beta,\gamma) z^n,\;where \;a_n^q(\alpha,\beta,\gamma)=\frac{(\gamma)_{nq}}{n!\Gamma(\alpha n+\beta)},\;n\geq0.$$
Clearly if $\gamma_1\geq\gamma_2>0,$ then $(\gamma_1)_{qn}\geq(\gamma_2)_{qn},$ and consequently $a_n^q(\alpha,\beta,\gamma_1)\geq a_n^q(\alpha,\beta,\gamma_2).$ Thus the function $\gamma\mapsto E_{\alpha,\beta}^{\gamma,q}(z)$ is increasing $(0,\infty)$.\\
\textbf{b.} For log-convexity property of the function $\beta\mapsto \mathbb{E}_{\alpha,\beta}^{\gamma,q}(z)$,  we observe that it is enough to show the
log-convexity of each individual term and to use the fact that sums of log-convex functions. Thus, we just need to show that for each $k\geq0$ we have
\[\frac{\partial^{2}}{\partial\beta^{2}}\log\left[\frac{\Gamma(\beta)}{\Gamma(\beta+\alpha k)}\right]=\psi^{'}(\beta)-\psi^{'}(\beta+\alpha k) \ge 0,\]
where $\psi(x)=\Gamma^{\prime}(x)/\Gamma(x)$ is the so--called digamma function. But $\psi$ is known to be concave, and consequently the function  $\beta\mapsto \frac{\Gamma(\beta)}{\Gamma(\beta+k\alpha)}$ is log--convex on $(0,\infty).$\\
\textbf{c.} Since the function $\beta\mapsto \mathbb{E}_{\alpha,\beta}^{\gamma,q}(z)$ is log--convex, then for all $\beta_1, \beta_2>0,\;z>0$ and $t\in[0,1]$ we have
$$ \mathbb{E}_{\alpha, t\beta_1+(1-t)\beta_2}^{\gamma,q}(z)\leq  [\mathbb{E}_{\alpha,\beta_1}^{\gamma,q}(z)]^t[\mathbb{E}_{\alpha,\beta_2}^{\gamma,q}(z)]^{1-t}.$$
Now choosing $t=1/2,\;\beta_1=\beta,\;\beta_2=\beta+2,$ we conclude that (\ref{666}) holds.
\end{proof}
\begin{theorem}Let $\alpha>0,0<\gamma\leq 1,$ and $\beta>x^{\star}-1,$ where $x^{\star}$ is the abscissa of the minimum of the $\Gamma$ function. Then the following Tur\'an type inequality
\begin{equation}\label{001}
E_{\alpha,\beta}^{\gamma,1}(z)E_{\alpha,\beta+2}^{\gamma,1}(z)\geq(E_{\alpha,\beta+1}^{\gamma,1}(z))^2-\frac{1}{\Gamma(\beta+1)\Gamma(\beta+2)(1-z)^2},
\end{equation}
holds for all $z\in[0,1[.$
\end{theorem}
\begin{proof}Let $\alpha>0,\;\beta>0$ and $\gamma>0.$ The Cauchy product reveals that
$$E_{\alpha,\beta}^{\gamma,1}(z)E_{\alpha,\beta+2}^{\gamma,1}(z)=\sum_{n=0}^{\infty}\left(\sum_{k=0}^n\frac{(\gamma)_{k}(\gamma)_{n-k}}{k!(n-k)!\Gamma(\alpha k+\beta)\Gamma(\alpha (n-k)+\beta+2)}\right)z^{2n},$$
and
$$(E_{\alpha,\beta+1}^{\gamma,1}(z))^2=\sum_{n=0}^{\infty}\left(\sum_{k=0}^n\frac{(\gamma)_k(\gamma)_{n-k}}{k!(n-k)!\Gamma(\alpha k+\beta+1)\Gamma(\alpha (n-k)+\beta+1)}\right)z^{2n}.$$
Thus,
\begin{equation}\label{99}
\begin{split}
E_{\alpha,\beta}^{\gamma,1}(z)E_{\alpha,\beta+2}^{\gamma,1}(z)-(E_{\alpha,\beta+1}^{\gamma,1}(z))^2&=\sum_{n=0}^{\infty}\left(\sum_{k=0}^n\frac{(\gamma)_k(\gamma)_{n-k}[\alpha(2k-n)-1]}{k!(n-k)!\Gamma(\alpha k+\beta+1)\Gamma(\alpha (n-k)+\beta+2)}\right)z^n\\
&=\sum_{n=0}^{\infty}\sum_{k=0}^n T_{n,k}^{\gamma,1}(\alpha,\beta) z^n-\sum_{n=0}^{\infty}\sum_{k=0}^n\frac{(\gamma)_k(\gamma)_{n-k}z^n}{k!(n-k)!\Gamma(\alpha k+\beta+1)\Gamma(\alpha (n-k)+\beta+2)}.
\end{split}
\end{equation}
where $T_{n,k}^{\gamma,1}(\alpha,\beta)=\frac{(\gamma)_k(\gamma)_{n-k}[\alpha(2k-n)]}{k!(n-k)!\Gamma(\alpha k+\beta+1)\Gamma(\alpha (n-k)+\beta+2)}.$ If $n$ is even, then
\begin{equation*}
\begin{split}
\sum_{k=0}^n T_{n,k}^{\gamma,1}(\alpha,\beta)&=\sum_{k=0}^{n/2-1} T_{n,k}^{\gamma,1}(\alpha,\beta)+\sum_{k=n/2+1}^{n} T_{n,k}^{\gamma,1}(\alpha,\beta)+T_{n,n/2}^{\gamma,1}(\alpha,\beta)\\
&=\sum_{k=0}^{n/2-1} (T_{n,k}^{\gamma,1}(\alpha,\beta)+T_{n,n-k}^{\gamma,1}(\alpha,\beta))\\
&=\sum_{k=0}^{[(n-1)/2]}(T_{n,k}^{\gamma,1}(\alpha,\beta)+T_{n,n-k}^{\gamma,1}(\alpha,\beta)),
\end{split}
\end{equation*}
where $[.]$ denotes the greatest integer function. Similarly, if $n$ is odd, then
$$\sum_{k=0}^n T_{n,k}^{\gamma,1}(\alpha,\beta)=\sum_{k=0}^{[(n-1)/2]}(T_{n,k}^{\gamma,1}(\alpha,\beta)+T_{n,n-k}^{\gamma,1}(\alpha,\beta)).$$
Simplifying, we find that
\begin{equation}\label{4}
\begin{split}
T_{n,k}^{\gamma,1}(\alpha,\beta)+T_{n,n-k}^{\gamma,1}(\alpha,\beta)&=
\frac{(\gamma)_k(\gamma)_{n-k}}{k!(n-k)!\Gamma(\alpha k+\beta+1)\Gamma(\alpha (n-k)+\beta+1)}\left[\frac{\alpha(2k-n)}{\alpha(n-k)+\beta+1}+\frac{\alpha(n-2k)}{\alpha k+\beta+1}\right]\\
&=\frac{\alpha^2(\gamma)_k(\gamma)_{n-k}(n-2k)^2}{k!(n-k)!\Gamma(\alpha k+\beta+2)\Gamma(\alpha (n-k)+\beta+2)}.
\end{split}
\end{equation}
Since the function $\gamma\mapsto(\gamma)_k$ is increasing on $(0,\infty)$, and since $\gamma\in(0,1)$, we conclude that $(\gamma)_k\leq(1)_k=k!,$ and consequently $\frac{(\gamma)_k(\gamma)_{n-k}}{k!(n-k)!}\leq 1$. In addition, the function $x\mapsto\Gamma(x)$ is increasing on $[x^\star,\infty)$ where $x^\star\approx1.461632144 . . . $ is the abscissa of the minimum of the Gamma function, this implies that $\frac{1}{\Gamma(\alpha k+\beta+1)\Gamma(\alpha (n-k)+\beta+2)}\leq\frac{1}{\Gamma(\beta+1)\Gamma(\beta+2)}.$ Consequently, we get
\begin{equation}\label{5}
\frac{(\gamma)_k(\gamma)_{n-k}}{k!(n-k)!\Gamma(\alpha k+\beta+1)\Gamma(\alpha (n-k)+\beta+2)}\leq\frac{1}{\Gamma(\beta+1)\Gamma(\beta+2)}.
\end{equation}
Combining  (\ref{99}), (\ref{4}) and (\ref{5}) we obtain (\ref{001}).
\end{proof}
\begin{theorem}\label{tt3}Let $\alpha>0,\;\beta_1,\beta_2>0, \gamma>0$, and $q\in(0,1)\cup\mathbb{N}.$ If $\beta_1<\beta_2$\;(resp. $\beta_2<\beta_1$), then the function $z\mapsto E_{\alpha,\beta_1}^{\gamma,q}(z)\big/E_{\alpha,\beta_2}^{\gamma,q}(z)$ is increasing (resp. decreasing) on $(0,\infty).$ Therefore, the following Tur\'an type inequalities
\begin{equation}\label{7777}
E_{\alpha,\beta_2}^{\gamma,q}(z)E_{\alpha,\beta_1-1}^{\gamma,q}(z)-E_{\alpha,\beta_1}^{\gamma,q}(z)E_{\alpha,\beta_2-1}^{\gamma,q}(z)+(\beta_2-\beta_1)E_{\alpha,\beta_1}^{\gamma,q}(z)E_{\alpha,\beta_2}^{\gamma,q}(z)\geq0,
\end{equation}
holds for all $\alpha,\beta_1,\beta_2>0$ such that $\beta_2>\beta_1.$ In particular, the Tur\'an type inequality
\begin{equation}\label{8888}
E_{\alpha,\beta}^{\gamma,q}(z)E_{\alpha,\beta+2}^{\gamma,q}(z)-[E_{\alpha,\beta+1}^{\gamma,q}(z)]^2+(\beta+1)E_{\alpha,\beta+1}^{\gamma,q}(z)E_{\alpha,\beta+2}^{\gamma,q}(z)\geq0,
\end{equation}
is valid for all $\alpha,\beta,z>0.$
\end{theorem}
\begin{proof}
From the power-series representation of the Mittag-Leffler function $E_{\alpha,\beta}^{\gamma,q}(z)$, we have
$$E_{\alpha,\beta_1}(z)\big/ E_{\alpha,\beta_2}(z)=\sum_{k=0}^\infty\frac{(\gamma)_{kq}z^k}{k!\Gamma(\beta_1+k\alpha)}\Bigg/\sum_{k=0}^\infty\frac{(\gamma)_{kq}z^k}{k!\Gamma(\beta_2+k\alpha)}.$$
In view of Lemma \ref{l2}, to prove the monotonicity properties of the function $E_{\alpha,\beta_1}^{\gamma,q}(z)\big/E_{\alpha,\beta_2}^{\gamma,q}(z)$ it is sufficient to prove the monotonicity of the sequence $u_k=\left\{\Gamma(\beta_2+k\alpha)/\Gamma(\beta_1+k\alpha)\right\}_{k\geq0}$.
 In \cite{KS}, the authors proved that the sequences $u_k$ is increasing  if and only if $\beta_1<\beta_2$ and  $u_k$ is decreasing  if and
only if $\beta_2<\beta_1$. Consequently,  if $\beta_1<\beta_2$, (resp. $\beta_2<\beta_1$), then the function $z\mapsto E_{\alpha,\beta_1}^{\gamma,q}(z)\Big /E_{\alpha,\beta_2}^{\gamma,q}(z)$ is increasing (resp. decreasing) on $(0,\infty).$ By again using the differentiation formula [\cite{SP}, Theorem 2.1]
\begin{equation}
\frac{d}{dz}E_{\alpha,\beta}^{\gamma,q}(z)=\frac{E_{\alpha,\beta-1}^{\gamma,q}(z)-(\beta-1)E_{\alpha,\beta}^{\gamma,q}(z)}{\alpha z}
\end{equation}
we obtain for $\beta_2>\beta_1$
$$\left[\frac{E_{\alpha,\beta_1}^{\gamma,q}(z)}{ E_{\alpha,\beta_2}^{\gamma,q}(z)}\right]^\prime=\frac{E_{\alpha,\beta_1-1}^{\gamma,q}(z)E_{\alpha,\beta_2}^{\gamma,q}(z)-E_{\alpha,\beta_2-1}^{\gamma,q}(z)E_{\alpha,\beta_1}^{\gamma,q}(z)+(\beta_2-\beta_1)E_{\alpha,\beta_1}^{\gamma,q}(z)E_{\alpha,\beta_2}^{\gamma,q}(z)}{\alpha z [E_{\alpha,\beta_2}^{\gamma,q}(z)]^2}\geq0.$$
So, the inequality (\ref{7777}) is proved. Now,  choosing $\beta_1=\beta+1$ and $\beta_2=\beta+2$ in the inequality (\ref{7777}) we obtain (\ref{8888}). The proof of Theorem \ref{tt3} is now completed.
\end{proof}
\begin{theorem}\label{t4}Let $n\in\mathbb{N},\alpha>0,\beta>0,$ and $\gamma>0$. We define the function $E_{\alpha,\beta}^{\gamma,q,n}(z)$ on $(0,\infty)$ by
$$E^{\gamma,1,n}_{\alpha, \beta}(z)=E^{\gamma,1}_{\alpha, \beta}(z)-\sum_{k=0}^{n}\frac{(\gamma)_{k}z^k}{k!\Gamma(\alpha k+\beta)}=\sum_{k=n+1}^{\infty}\frac{(\gamma)_{k}z^k}{k!\Gamma(\alpha k+\beta)}.$$
Then the following Tur\'an type inequality
\begin{equation}\label{KK1}
E^{\gamma,1,n}_{\alpha, \beta}(z)E^{\gamma,1,n+2}_{\alpha, \beta}(z)\leq (E^{\gamma,1,n+1}_{\alpha, \beta}(z))^2,
\end{equation}
is valid.
\end{theorem}
\begin{proof} From the definition of the function $E^{\gamma,1,n}_{\alpha, \beta}(z)$ we get
$$E^{\gamma,1,n}_{\alpha, \beta}(z)=E^{\gamma,1,n+1}_{\alpha, \beta}(z)+\frac{(\gamma)_{(n+1)}z^{n+1}}{(n+1)!\Gamma(\alpha (n+1)+\beta)}$$
and
$$E^{\gamma,1,n+2}_{\alpha, \beta}(z)=E^{\gamma,1,n+1}_{\alpha, \beta}(z)-\frac{(\gamma)_{(n+2)}z^{n+2}}{(n+2)!\Gamma(\alpha (n+2)+\beta)}.$$
Thus
$$E^{\gamma,1,n}_{\alpha, \beta}(z)E^{\gamma,1,n+2}_{\alpha, \beta}(z)-(E^{\gamma,1,n+1}_{\alpha, \beta}(z))^2=$$
\begin{equation*}
\begin{split}
&=E^{\gamma,1,n+1}_{\alpha, \beta}(z)\left(\frac{(\gamma)_{(n+1)}z^{n+1}}{(n+1)!\Gamma(\alpha (n+1)+\beta)}-\frac{(\gamma)_{(n+2)}z^{n+2}}{(n+2)!\Gamma(\alpha (n+2)+\beta)}\right)-\frac{(\gamma)_{(n+1)}(\gamma)_{(n+2)}z^{2n+3}}{(n+1)!(n+2)!\Gamma(\alpha (n+1)+\beta)\Gamma(\alpha (n+2)+\beta)}\\
&=\frac{(\gamma)_{(n+1)}z^{n+1}E^{\gamma,1,n+2}_{\alpha, \beta}(z)}{(n+1)!\Gamma(\alpha (n+1)+\beta)}-\frac{(\gamma)_{(n+2)}z^{n+2}E^{\gamma,1,n+1}_{\alpha, \beta}(z)}{(n+2)!\Gamma(\alpha (n+2)+\beta)}\\
&=\frac{(\gamma)_{(n+1)}}{(n+1)!\Gamma(\alpha (n+1)+\beta)}\sum_{k=n+3}^\infty\frac{(\gamma)_{k}z^{k+n+1}}{k!\Gamma(\alpha k+\beta)}-\frac{(\gamma)_{(n+2)}}{(n+2)!\Gamma(\alpha (n+2)+\beta)}\sum_{k=n+2}^\infty\frac{(\gamma)_{k}z^{k+n+2}}{k!\Gamma(\alpha k+\beta)}\\
&=\sum_{k=n+3}^\infty\left[\frac{(\gamma)_{(n+1)}(\gamma)_{k}}{(n+1)!\Gamma(\alpha (n+1)+\beta)k!\Gamma(\alpha k+\beta)}-\frac{(\gamma)_{(n+2)}(\gamma)_{(k-1)}}{(n+2)!\Gamma(\alpha (n+2)+\beta)(k-1)!\Gamma(\alpha (k-1)+\beta)}\right]z^{k+n+1}\\
&=\sum_{k=n+3}^\infty\left[\frac{A^{\gamma,1}_k(\alpha,\beta)}{k!(n+2)!\Gamma(\alpha (k-1)+\beta)\Gamma(\alpha k+\beta)\Gamma(\alpha (n+2)+\beta)\Gamma(\alpha (n+1)+\beta)}\right]z^{k+n+1},
\end{split}
\end{equation*}
where \\
$$A_{k}^{\gamma,1}(\alpha,\beta)=(n+1)(\gamma)_{n+1}(\gamma)_{k}\Gamma(\alpha (n+2)+\beta)\Gamma(\alpha (k-1)+\beta)-(k-1)(\gamma)_{n+2}(\gamma)_{(k-1)}\Gamma(\alpha (n+1)+\beta)\Gamma(\alpha k+\beta).$$
On the other hand, we have
\begin{equation}\label{MM}
\begin{split}
A_{k}^{\gamma,1}(\alpha,\beta)&=(n+1)(\gamma)_{n+1}(\gamma)_{k}\left[\Gamma(\alpha (n+2)+\beta)\Gamma(\alpha (k-1)+\beta)-\frac{(k-1)(\gamma)_{n+2}(\gamma)_{k-1}}{(n+1)(\gamma)_{n+1}(\gamma)_{k}}\Gamma(\alpha (n+1)+\beta)\Gamma(\alpha k+\beta)\right]\\
&=(n+1)(\gamma)_{n+1}(\gamma)_{k}\left[\Gamma(\alpha (n+2)+\beta)\Gamma(\alpha (k-1)+\beta)-\frac{\Gamma(\gamma+n+2)\Gamma(\gamma+k-1)}{\Gamma(\gamma+n+1)\Gamma(\gamma+k)}\Gamma(\alpha (n+1)+\beta)\Gamma(\alpha k+\beta)\right]\\
&\leq(n+1)(\gamma)_{n+1}(\gamma)_{k}\Bigg[\Gamma(\alpha (n+2)+\beta)\Gamma(\alpha (k-1)+\beta)-\Gamma(\alpha (n+1)+\beta)\Gamma(\alpha k+\beta)\Bigg].
\end{split}
\end{equation}
Taking into account the inequality [\cite{KS}, p. 4]
\begin{equation}
\frac{\Gamma(\alpha (k+2)+\beta)}{\Gamma(\alpha (k+1)+\beta)}\leq\frac{\Gamma(\alpha k+\beta)}{\Gamma(\alpha (k-1)+\beta)},
\end{equation}
which holds for all $\alpha,\beta>0$ and $k\geq n+3,$ and using (\ref{MM}), clearly we have $A_{k}^{\gamma,1}(\alpha,\beta)\leq0.$ This in turn implies that the inequality (\ref{KK1}) hold.
\end{proof}
\begin{theorem}\label{t5}Let $n\in\mathbb{N},\alpha>0,\beta>0,$ and $\gamma>0$. We define the function $H_{\alpha,\beta}^{\gamma,1,n}$ on $(0,\infty),$ by
\begin{equation}
H_{\alpha,\beta}^{\gamma,1}(z)=\frac{E^{\gamma,1,n}_{\alpha, \beta}(z)E^{\gamma,1,n+2}_{\alpha, \beta}(z)}{\Big(E^{\gamma,1,n+1}_{\alpha, \beta}(z)\Big)^2},\;\;z>0.
\end{equation}
Then the function $x\mapsto H_{\alpha,\beta}^{\gamma,1}(z)$ is increasing on $(0,\infty).$ So, the following Tur\'an type inequality
\begin{equation}\label{SS}
\frac{(n+1)(\gamma+n+2)}{(n+2)(\gamma+n+1)}.\frac{\Gamma^2(\beta+\alpha(n+2))}{\Gamma(\beta+\alpha(n+1))\Gamma(\beta+\alpha(n+3))}\Big(E^{\gamma,1,n+1}_{\alpha, \beta}(z)\Big)^2\leq E^{\gamma,1,n}_{\alpha, \beta}(z)E^{\gamma,1,n+2}_{\alpha, \beta}(z),
\end{equation}
valid for all $n\in\mathbb{N},\alpha>0,\beta>0,$ and $\gamma>0$. The constant in left hand side of inequality (\ref{SS}) is sharp.
\end{theorem}
\begin{proof} By again using the Cauchy product we gave
\begin{equation*}
\begin{split}
H_{\alpha,\beta}^{\gamma,1}(z)&=\frac{\sum_{k=0}^\infty\left(\sum_{j=0}^k\frac{(\gamma)_{j+n+1}(\gamma)_{k-j+n+3}}{(j+n+1)!(k-j+n+3)!\Gamma(\beta+\alpha(j+n+1))\Gamma(\beta+\alpha(k-j+n+3))}\right)z^{2n+2k+4}}{\sum_{k=0}^\infty\left(\sum_{j=0}^k\frac{(\gamma)_{j+n+2}(\gamma)_{k-j+n+2}}{(k-j+n+2)!(j+n+2)!\Gamma(\beta+\alpha(j+n+2))\Gamma(\beta+\alpha(k-j+n+2))}\right)z^{2n+2k+4}}.
\end{split}
\end{equation*}
We define the sequence $U_j(\alpha,\beta,\gamma)$, by
$$V_j(\alpha,\beta,\gamma)=\frac{(\gamma)_{j+n+1}(\gamma)_{k-j+n+3}(j+n+2)!(k-j+n+2)!}{(\gamma)_{j+n+2}(\gamma)_{k-j+n+2}(j+n+1)!(k-j+n+3)!}U_j(\alpha,\beta),$$
where  the sequence $U_j(\alpha,\beta)$ is defined by \cite{KS}
$$U_j(\alpha,\beta)=\frac{\Gamma(\beta+(n+2+j)\alpha)\Gamma(\beta+(n+2+k-j)\alpha)}{\Gamma(\beta+(n+1+j)\alpha)\Gamma(\beta+(n+3+k-j)\alpha)}.$$
In \cite{KS}, the authors proved that the sequence $U_j(\alpha,\beta)$ is increasing for all $j=0,1,...$ Thus
\begin{equation*}
\begin{split}
\frac{V_{j+1}(\alpha,\beta,\gamma)}{V_j(\alpha,\beta,\gamma)}&=K_{k,j}(\alpha,\beta,\gamma)\frac{U_{j+1}(\alpha,\beta)}{U_j(\alpha,\beta)}\\
&\geq K_{k,j}(\alpha,\beta,\gamma),
\end{split}
\end{equation*}
with
$$K_{k,j}(\alpha,\beta,\gamma)=\frac{(\gamma+k-j+n+1)(j+n+2)(k-j+n+2)(\gamma+j+n+1)}{(\gamma+j+n+2)(k-j+n+1)(j+n+1)(\gamma+k-j+n+2)}.$$
By using the inequality
$$\frac{(\gamma+j+n+1)(j+n+2)}{(\gamma+j+n+2)(j+n+1)}\geq1,$$
we have
\begin{equation*}
K_{k,j}(\alpha,\beta,\gamma)=\left[\frac{(\gamma+k-j+n+1)(k-j+n+2)}{(\gamma+k-j+n+2)(k-j+n+1)}\right]\left[\frac{(\gamma+j+n+1)(j+n+2)}{(\gamma+j+n+2)(j+n+1)}\right]\geq1.
\end{equation*}
So, the sequence $V_{j}(\alpha,\beta,\gamma)$ is increasing for $j=0,1,...,$ and $\alpha,\beta,\gamma>0.$ This implies that the ratios
$$\frac{\sum_{j=0}^k\frac{(\gamma)_{j+n+1}(\gamma)_{k-j+n+3}}{(j+n+1)!(k-j+n+3)!\Gamma(\beta+\alpha(j+n+1))\Gamma(\beta+\alpha(k-j+n+3))}}{\sum_{j=0}^k\frac{(\gamma)_{j+n+2}(\gamma)_{k-j+n+2}}{(k-j+n+2)!(j+n+2)!\Gamma(\beta+\alpha(j+n+2))\Gamma(\beta+\alpha(k-j+n+2))}},$$
is increasing, by Lemma 1. So the function $x\mapsto H_{\alpha,\beta}^{\gamma,1}(z)$ is increasing by Lemma 2. Finally, it is easy to see that
$$\lim_{x\rightarrow0}H_{\alpha,\beta}^{\gamma,1}(z)=\frac{(n+1)(\gamma+n+2)}{(n+2)(\gamma+n+1)}.\frac{\Gamma^2(\beta+\alpha(n+2))}{\Gamma(\beta+\alpha(n+1))\Gamma(\beta+\alpha(n+3))}.$$
So, the constant $\frac{(n+1)(\gamma+n+2)}{(n+2)(\gamma+n+1)}.\frac{\Gamma^2(\beta+\alpha(n+2))}{\Gamma(\beta+\alpha(n+1))\Gamma(\beta+\alpha(n+3))}$ is the best possible for which the inequality (\ref{SS}) holds for all $\alpha,\beta,\gamma>0,\;n\in\mathbb{N}$ and $z>0.$
\end{proof}
\section{Lazarevi\'c and Wilker type inequalities for the generalized Mittag--Leffler functions}
\begin{theorem}\label{t6} Let $\alpha,\beta_1,\beta_2,\gamma>0$ and $n\in\mathbb{N}$. If $\beta_1\geq\beta_2,$ then the Lazarevi\'c-type inequality for generalized Mittag--Leffler functions
\begin{equation}\label{07}
\Big[\mathbb{E}^{\gamma,q}_{\alpha,\beta_1+1}(z)\Big]^{\beta_2}\leq\Big[\mathbb{E}^{\gamma,q}_{\alpha,\beta_2+1}(z)\Big]^{\beta_1}
\end{equation}
holds for all $z\in\mathbb{R}.$ In particular, the following inequality
\begin{equation}\label{007}
\mathbb{E}^{\gamma,q}_{\alpha,5/2}(z)\leq\Big[\mathbb{E}^{\gamma,q}_{\alpha,3/2}(z)\Big]^{3},
\end{equation}
holds for all $z\in\mathbb{R},$ $\alpha,\gamma>0$ and $n\in\mathbb{N}$.
\end{theorem}
\begin{proof}Since the function $\beta\mapsto\mathbb{E}^{\gamma,q}_{\alpha,\beta}(z)$ is log--convex on $(0,\infty),$ thus the function $\beta\mapsto\frac{\mathbb{E}^{\gamma,q}_{\alpha,\beta+1}(z)}{\mathbb{E}^{\gamma,q}_{\alpha,\beta}(z)}$ is increasing on $(0,\infty).$ For $\beta_1\geq\beta_2>0,$ we define the function $\Phi$ by
$$\Phi(z)=\frac{\beta_2}{\beta_1}\log\big(\mathbb{E}^{\gamma,q}_{\alpha,\beta_1+1}(z)\big)-\log\big(\mathbb{E}^{\gamma,q}_{\alpha,\beta_2+1}(z)\big).$$
From the differentiation formula [\cite{SP}, Theorem 2.1]
$$E^{\gamma,q}_{\alpha,\beta}(z)=\beta E^{\gamma,q}_{\alpha,\beta+1}(z)+\alpha z\frac{d}{dz}E^{\gamma,q}_{\alpha,\beta+1}(z),$$
we obtain that
\begin{equation*}
\begin{split}
\Phi^\prime(z)&=\frac{\beta_2\bigg(E^{\gamma,q}_{\alpha,\beta_1}(z)-\beta_1E^{\gamma,q}_{\alpha,\beta_1+1}(z)\bigg)}{\alpha\beta_1 zE^{\gamma,q}_{\alpha,\beta_1+1}(z)}-\frac{E^{\gamma,q}_{\alpha,\beta_2}(z)-\beta_2E^{\gamma,q}_{\alpha,\beta_2+1}(z)}{\alpha z E^{\gamma,q}_{\alpha,\beta_2+1}(z)}\\
&=\frac{\beta_2}{\alpha z}\left(\frac{E^{\gamma,q}_{\alpha,\beta_1}(z)}{\beta_1E^{\gamma,q}_{\alpha,\beta_1+1}(z)}-\frac{E^{\gamma,q}_{\alpha,\beta_2}(z)}{\beta_2E^{\gamma,q}_{\alpha,\beta_2+1}(z)}\right)\\
&=\frac{\beta_2}{\alpha z}\left(\frac{\mathbb{E}^{\gamma,q}_{\alpha,\beta_1}(z)}{\mathbb{E}^{\gamma,q}_{\alpha,\beta_1+1}(z)}-\frac{\mathbb{E}^{\gamma,q}_{\alpha,\beta_2}(z)}{\mathbb{E}^{\gamma,q}_{\alpha,\beta_2+1}(z)}\right).
\end{split}
\end{equation*}
Thus, for all $\beta_1\geq\beta_2>0,$ we conclude that the function $z\mapsto \Phi(z)$ is decreasing on $[0,\infty),$ and decreasing on $(-\infty,0]$. Therefore, for all $z\in\mathbb{R}$, we have $\Phi(z)\leq\Phi(0)=0.$ Now,  Choosing $\beta_1=3/2$ and $\beta_1=1/2$ in (\ref{07}) we obtain (\ref{007}). This completes the proof of Theorem \ref{t6}.
\end{proof}
\begin{corollary}Let $\alpha,\beta_1,\beta_2,\gamma>0$ and $n\in\mathbb{N}$. If $\beta_1\geq\beta_2,$ then the Wilker type inequality for generalized Mittag--Leffler functions
\begin{equation}\label{08}
\big[\mathbb{E}^{\gamma,q}_{\alpha,\beta_2+1}(z)\big]^\frac{\beta_1-\beta_2}{\beta_2}+\frac{\mathbb{E}^{\gamma,q}_{\alpha,\beta_2+1}(z)}{\mathbb{E}^{\gamma,q}_{\alpha,\beta_1+1}(z)}\geq2,
\end{equation}
holds for all $z\in\mathbb{R}.$ In particular, the following inequality
\begin{equation}\label{09}
\big[\mathbb{E}^{\gamma,q}_{\alpha,3/2}(z)\big]^2+\frac{\mathbb{E}^{\gamma,q}_{\alpha,3/2}(z)}{\mathbb{E}^{\gamma,q}_{\alpha,5/2}(z)}\geq2,
\end{equation}
is valid for all $z\in\mathbb{R},$ $\alpha,\gamma>0$ and $n\in\mathbb{N}$.
\end{corollary}
\begin{proof}From the inequality (\ref{07}), we get
\begin{equation}
\mathbb{E}^{\gamma,q}_{\alpha,\beta_1+1}(z)\leq\mathbb{E}^{\gamma,q}_{\alpha,\beta_2+1}(z).\left[\mathbb{E}^{\gamma,q}_{\alpha,\beta_2+1}(z)\right]^{\frac{\beta_1-\beta_2}{\beta_2}},
\end{equation}
combining this inequality and the arithmetic--geometric mean inequality, we conclude that (\ref{08}) holds. Finally, let $\beta_1=3/2$ and $\beta_2=1/2$, we get (\ref{09}).
\end{proof}

\section{Open Problems}

Motivated by Theorems \ref{t4}  and \ref{t5} we pose the following problems.

\textbf{Problem 1.} Motivated by the inequality (\ref{KK1}) in Theorem \ref{t4} we pose the following problem: find the generalization of the inequality (\ref{KK1}) in the following inequality
\begin{equation}
E^{\gamma,q,n}_{\alpha, \beta}(z)E^{\gamma,q,n+2}_{\alpha, \beta}(z)\leq (E^{\gamma,q,n+1}_{\alpha, \beta}(z))^2,
\end{equation}
where $\alpha,\beta,\gamma>0,q\in(0,1)\cup\mathbb{N}$ and $z>0.$

\textbf{Problem 2.} For $z\in(0,\infty),$ find the monotonicity of the function
\begin{equation}
H_{\alpha,\beta}^{\gamma,q,n}(z)=\frac{E^{\gamma,q,n}_{\alpha, \beta}(z)E^{\gamma,q,n+2}_{\alpha, \beta}(z)}{\Big(E^{\gamma,q,n+1}_{\alpha, \beta}(z)\Big)^2}
\end{equation}
for all $n\in\mathbb{N},\alpha>0,\beta>0,\gamma>0$ and $q\in(0,1)\cup\mathbb{N}.$


\end{document}